\documentclass[10pt]{amsart}

\usepackage{amsmath,amssymb,amscd,amsxtra,amsfonts,mathrsfs}
\usepackage{epsf,graphicx,epsfig,cite,latexsym}
\usepackage{mathtools}
\usepackage{latexsym, amsmath, amssymb, a4, epsfig}
\usepackage[T1]{fontenc}

\usepackage{algorithm}
\usepackage{algorithmic}

 \usepackage{mathrsfs}
\newcommand{\no}[1]{#1}
\renewcommand{\no}[1]{}
\no{\usepackage{times}\usepackage[subscriptcorrection, slantedGreek, nofontinfo]{mtpro}
\renewcommand{\Delta}{\upDelta}}
\usepackage{color}

\date{\today}
\def\nm{\noalign{\medskip}}
\numberwithin{algorithm}{section}
\numberwithin{figure}{section}

\newtheorem{lemma}{Lemma}[section]
\newtheorem{remark}{Remark}[section]
\newtheorem{definition}{Definition}[section]
\newtheorem{corollary}{Corollary}[section]
\newtheorem{proposition}{Proposition}[section]
\newtheorem{theorem}{Theorem}[section]
\newtheorem{example}{Example}[section]

\def\R{\mathbb{R}}

\def\N{\mathbb{N}}
\def\la{\langle}
\def\ra{\rangle}
\newcommand{\be}{\begin{equation}}
\newcommand{\ee}{\end{equation}}
\newcommand{\ba}{\begin{array}}
\newcommand{\ea}{\end{array}}

\newcommand{\p}{\partial}
\newcommand{\bea}{\begin{eqnarray*}}
\newcommand{\eea}{\end{eqnarray*}}
\newcommand{\bean}{\begin{eqnarray}}
\newcommand{\eean}{\end{eqnarray}}

\def\tilde{\widetilde}

\def\cydot{\leavevmode\raise.4ex\hbox{.}}

\title[]{Identification of an algebraic domain in two dimensions from a finite number 
of its  generalized polarization tensors }

\author{Habib Ammari }

\address{Habib Ammari, Department of Mathematics, ETH Z\"urich, 
R\"amistrasse 101, 8092 Z\"urich, Switzerland}
\email{habib.ammari@math.ethz.ch}

\author{Mihai Putinar}

\address{Mihai Putinar, Department of Mathematics, University of California at Santa Barbara, Santa Barbara, CA 93106-3080, USA,  and School of Mathematics \& Statistics, Newcastle University Newcastle upon Tyne, NE1 7RU, 
United Kingdom}

\email{mputinar@math.ucsb.edu}

\author{Andries Steenkamp} 
\address{Andries Steenkamp, Department of Mathematics, ETH Z\"urich, 
R\"amistrasse 101, 8092 Z\"urich, Switzerland}

\email{andriess@student.ethz.ch}

\author{Faouzi Triki}

\address{Faouzi Triki,  Laboratoire Jean Kuntzmann,  UMR CNRS 5224, 
Universit\'e  Grenoble-Alpes, 700 Avenue Centrale,
38401 Saint-Martin-d'H\`eres, France}

\email{faouzi.triki@univ-grenoble-alpes.fr}

\thanks{
The work of Faouzi Triki was supported in part by the
 grant ANR-17-CE40-0029 of the French National Research Agency ANR (project MultiOnde),
 and the LabEx PERSYVAL-Lab (ANR-11-LABX- 0025-01).}

\date{\today}
\begin{document}

\begin{abstract}
This paper aims at studying how finitely many generalized polarization tensors of an algebraic domain
can be used to determine its shape. Precisely,  given a planar set with real algebraic boundary,  it  is 
shown that the minimal polynomial with real coefficients vanishing on the boundary can be identified
as the generator of a one dimensional kernel of a matrix whose entries are obtained from a finite  number
of generalized  polarization tensors. The size of the matrix depends  polynomially on the degree of the boundary of the algebraic domain. The density with respect to Hausdorff distance of algebraic domains among all bounded domains invites to extend via approximation our reconstruction procedure beyond its natural context. Based on this, a new algorithm for shape recognition/classification is proposed with some strong hints about its efficiency. 
\end{abstract}

\maketitle
\tableofcontents 

\section{Introduction}\label{sec0}
Given a conductivity contrast, the generalized polarization tensors (GPTs)  of a bounded Lipschitz domain are an infinite sequence of
tensors.  The GPTs form the basic building blocks for the far-field behavior of  the electric potential.
 Recently, many works have shown that  the GPTs
can be used to efficiently recover geometrical properties of the underlying shape.  In fact the  knowledge of the full set of 
GPTs  determines uniquely  the shape of the domain as proved
 in \cite{ak03}.   When the domain is transformed by a rigid motion or a dilation,
the corresponding GPTs change according to certain rules. It is
possible to construct as combinations of GPTs invariants under
these transformations. This property makes GPTs suitable for the
dictionary matching problem~\cite{Hu, AmmariBoulierGarnierJingKangWang}. 
The GPTs have also been used in various areas of applications 
such as imaging, cloaking, and plasmonics. We refer the reader to \cite{AmmariKangb1, ak11, cloaking, plasmonics, kang, milton, TrikiVauthrin} and 
 references therein for further information about  these applications. 

Since the GPTs appear naturally in imaging a small conductivity inclusion from 
boundary potential measurements,  the amount of information about the shape of the inclusion  
encoded in the first tensors is richer than any other geometrical quantities.
Recent numerical studies \cite{AGKLY, AKLZ} show that
by using only the first few terms of GPTs reasonable approximation of the true 
shape can be recovered. Complete geometric identification of a conductivity
 inclusion from the knowledge  of its first  GPT  is  known to be  possible for ellipsoid 
 shapes. In fact if the contrast is given, only the first polarization tensor is needed to retrieve the major and minor axis of an ellipse~\cite{AmmariKangb1}. For arbitrary shapes, it is proposed to approach them using  
 ellipse-equivalent identification. This consists simply of determining the shape of an ellipse with the same first polarization tensor as that of the targeted  inclusion~\cite{AmmariKangb1}.
The results of this approach are quite surprising since the recovered equivalent ellipse 
seems to hold much more  information on the shape than anticipated. For example, the equivalent ellipse contains more knowledge than the first two or three Dirichlet Laplacian eigenvalues of the inclusion. An interesting question  is whether one could  recover other shapes of  an inclusion from  the knowledge  of  a finite number of its GPTs. In view of \cite{jmpa},  inclusions with algebraic shapes represent  good candidates for such an identification problem. To specify our terminology, an inclusion has algebraic shape if it is a bounded open subset of Euclidean space whose boundary is real algebraic, i.e., contained in the zero set of finitely many polynomials.

In this paper,  we are interested in the inverse problem of 
recovering the shape of an algebraic inclusion given a finite number of its GPTs. We consider shapes unique up to rigid motions, that is orthogonal transformations.

The  paper is organized as follows. In Section \ref{sec1} we introduce the notion of 
GPTs of an inclusion and their relation to far-field expansion of the 
fields associated to a piecewise constant conductivity. The inverse problem in question is stated in Section  \ref{sec2}, where a review of recent results  in the recovery  of the shape 
of an inclusion  from the knowledge of all the GPTs is also given. Section \ref{sec3}  is dedicated to an introduction to real algebraic domains. Some basic notions of real algebraic geometry are recalled here.  Our main identification
result is stated  in Theorem \ref{main}. The detailed proof of the main result as well as  a uniqueness result are provided in  Section  \ref{sec4}.    Section \ref{sec6} is  devoted  to the generalization of 
the concept of  ellipse-equivalent approach to higher-order GPTs. This generalization takes advantage of the density of algebraic domains in the set of smooth inclusions. 
We apply the main result in Section \ref{sec6} by constructing a shape recognition algorithm and demonstrating its optimality by means  of a few well chosen examples.
The paper is concluded with some discussions in Section \ref{sec7}.

\section{Generalized polarization tensors}\label{sec1}

Let $D$ be a bounded  Lipschitz domain in $\R^2,$ of size of order one. Assume that its boundary 
 $\partial D$ contains the origin. Throughout this paper, we use standard notation concerning Sobolev spaces. For a density 
$\phi \in H^{-1/2}(\partial D)$,  define the  Neumann-Poincar\'e operator (NPO):  
$\mathcal{K}^*_{\p D}: \, H^{-1/2}(\partial D) \rightarrow H^{-1/2}(\partial D), $
by
\bea \label{introkd2}
\mathcal{K}^*_{\p D} [\phi] (x) = \frac{1}{2\pi} \mbox{p.v.} \, \int_{\p D}
\frac{\la x -y, \nu(x) \ra}{|x-y|^2} \phi(y)\,d\sigma(y), \quad x \in \p D, 
\eea
where $\mbox{p.v.}$ denotes the principal value, $\nu(x)$ is the outward unit normal to $\partial D$ at $x \in \partial D, $ and 
$\la\,, \,\ra$ denotes the scalar product in $\R^2$.


The spectral properties of the Neumann-Poincar\'e operator 
have proven interesting in several contexts~\cite{AndoKang,AmmariCiraoloKangLeeMilton,BDT,BonnetierTriki,BonnetierTriki_2}.
Due to Plemelj-Calder\'on identity  and energy estimates,  the spectrum of 
${\mathcal K}^*_{\p D}$  is real~\cite{AmmariKangb1,KhavisonPutinarShapiro}.
When $D$ is smooth (with ${\mathcal C}^{1,\alpha}$ boundary), 
${\mathcal K}^*_{\p D}$ is compact, hence its spectrum consists of a sequence of
eigenvalues that accumulates to $0$~\cite{KhavisonPutinarShapiro}. 
When $D$ is Lipschitz, the following proposition characterizes  the resolvent set  $\rho({\mathcal K}^*_{\p D})$
of the NPO~\cite{seo, AmmariKangb1}.
\begin{proposition} \label{resolvent} We have
$\mathbb C\setminus 
(-1/2, 1/2] \subset \rho({\mathcal K}^*_{\p D})$. Moreover, if $|\lambda | 
\geq 1/2$ , then  $(\lambda I - \mathcal{K}^*_{\p D}) $ is invertible on 
$H_0^{-1/2}(\p D):= \{f \in H^{-1/2}(\p D): \langle f, 1\rangle_{-1/2,1/2} =0\}$. Here, $\langle \;, \; \rangle_{-1/2,1/2}$ denotes the duality pairing between $H^{-1/2}(\partial D)$ and $H^{1/2}(\partial D)$.  
\end{proposition} 

For $|\lambda|>1/2$ and  a multi-index $\alpha = (\alpha_1, \alpha_2) \in \mathbb N^2$, 
 where $\mathbb N$ is the set of all
positive integers, define $\phi_\alpha$ by

\bea
\phi_\alpha(y) :=(\lambda I- \mathcal{K}^*_{\p D})^{-1} \left [\nu(x)\cdot\nabla x^\alpha\right ](y), \quad 
y \in \p D.
\eea

Here and throughout this paper, we use the conventional notation: $x^\alpha= x_1^{\alpha_1}
x_2^{\alpha_2},\; \alpha= (\alpha_1, \alpha_2) \in \N^2$, and $|\alpha| = \alpha_1+\alpha_2$. 
We also use the graded lexicographic order: $\alpha,\, \beta \in \N^2$ verifies
$\alpha \leq \beta$ if $|\alpha| < |\beta|$, or, if $|\alpha| = |\beta|$, then $\alpha_1\leq \beta_1$  or $\alpha_1=\beta_1$ and  $\alpha_2\leq \beta_2.$

The GPTs $M_{\alpha\beta}$ for $\alpha, \beta\in \N^2 \; (|\alpha|, |\beta| \geq 1)$, associated with the parameter $\lambda$ and the domain $D$ are defined by
 \bean  \label{gpt1}
 M_{\alpha\beta}(\lambda, D) :=  \int_{\p D} y^\beta \phi_\alpha(y)\,d\sigma(y).
 \eean

As we said before,  the GPTs  are tensors  that appear naturally  in the asymptotic expansion of  the electrical potential in the presence of a small inclusion $D$ of conductivity contrast $0\leq k \not=1 \leq +\infty$. The parameter $\lambda$ is related to the conductivity  $k$ via the relation
\bea
\lambda = \frac{k+1}{2(k-1)}.
\eea
The fact that $0 \leq k \leq +\infty$ implies that $|\lambda| \geq 1/2$, and hence $(\lambda I- \mathcal{K}^*_{\p D}) $ is invertible on 
$H_0^{-1/2}(\p D)$.  Assume now that the distribution of the conductivity in $\R^2$ is given by
\bea
\Upsilon=k \chi(D) + \chi(\R^2 \setminus \overline{D}),
 \eea
where $\chi$ denotes the indicator function. For a given harmonic function $h$ in $\R^2$,
 we consider the following transmission problem:
\bea \label{trans}
\left\{
\begin{array}{ll}
\nabla \cdot \Upsilon \nabla u = 0 \quad &\mbox{in } \R^2, \\
u(x)-h(x) = O(|x|^{-1}) \quad &\mbox{as } |x| \to \infty.
\end{array}
\right.
\eea
The electric potential  $u$  has the following integral representation (see, for instance,  \cite{AmmariKangb1})
\bea
u(x) = h(x) +\mathcal S_{\p D}(\lambda I- \mathcal{K}^*_{\p D})^{-1} \left[ \p_{\nu} h(x)\right ],\quad 
x\in \R^2,
\eea
where $\mathcal S_{\p D}: H^{-1/2}(\p D) \rightarrow H^{1/2}(\p D)$ is the single layer potential given
by
\bea
\mathcal{S}_{\p D} [\phi] (x) = \frac{1}{2\pi} \int_{\p D}\Gamma (x-y)
\phi(y)\,d\sigma(y), \quad x \in \R^2.
\eea
Here, $\Gamma$ is the fundamental solution of the Laplacian,
\bea
\Gamma(x):= \frac{1}{2\pi}\ln |x|.
\eea
It possesses  the following Taylor expansion 
\bea
\Gamma (x-y) = \sum_{|\alpha|=0}^{\infty}\frac{(-1)^{|\alpha|}}{\alpha!}\p^\alpha 
\Gamma(x)y^\alpha, \quad  y\in \overline D, \quad |x| \rightarrow +\infty,
\eea
where  $\partial_\alpha = \partial^{\alpha_1}_{x_1}\partial^{\alpha_2}_{x_2},$ and
$\alpha! = \alpha_1! \alpha_2!$ for $\alpha= (\alpha_1, \alpha_2) \in \N^2$.

Then, the far-field perturbation of the voltage potential created by  $D$ is given by  \cite{AmmariKangb1}
\bean \label{asym}
u(x)-h(x) = \sum_{|\alpha|, |\beta| =1}^{\infty}\frac{(-1)^{|\alpha|}}{\alpha!\beta!}\p^\alpha \Gamma(x)
M_{\alpha\beta}\p^\beta h(0) \qquad \textrm{as   }  |x| \rightarrow +\infty.
\eean

From the asymptotic expansion \eqref{asym},  we deduce that the knowledge of
 $M_{\alpha\beta}$ for $\alpha, \beta\in \N^2 \; (|\alpha|, |\beta| \geq 1)$ 
is equivalent to knowing the far-field responses of the inclusion for all harmonic excitations.
\section{Shape reconstruction problem }\label{sec2}
In this section,  a brief review of recent results  in GPT based inclusion shape recovery is given. We first introduce the harmonic combinations of the GPTs. Positivity and symmetry properties of the GPTs are proved using their harmonic combinations  \cite{AmmariKangb1}. 
 A harmonic combination of the  GPTs $M_{\alpha \beta}$ is 
  \bea
 \sum_{\alpha, \beta}a_\alpha b_\beta  M_{\alpha \beta},
 \eea where  $\sum_{\alpha}a_\alpha x^\alpha$ and  $\sum_{\beta} b_\beta x^\beta$ are real harmonic polynomials. We further call such $a_\alpha$ and $b_\beta$  harmonic coefficients. For example, if $a_\alpha$ and $b_\beta$ are any two  harmonic coefficients, we have the following symmetry property:
 \bea
 \sum_{\alpha, \beta}a_\alpha b_\beta  M_{\alpha \beta} = \sum_{\alpha, \beta}a_\alpha b_\beta  M_{\beta \alpha}.
 \eea
The  following uniqueness result has been proved in  \cite{ak03}.
\begin{theorem}\label{thm1}
If all harmonic combinations of GPTs of two domains $D$ and $\widetilde D$
with parameters $\lambda$ and $\tilde \lambda$, are identical, that is
\bea
 \sum_{\alpha, \beta}a_\alpha b_\beta  M_{\alpha \beta}(\lambda, D) = \sum_{\alpha, \beta}a_\alpha b_\beta  M_{\alpha \beta}(\tilde \lambda, \widetilde D),
 \eea
for all pairs  $a_\alpha$ and $b_\beta$ of harmonic coefficients, then $D= \widetilde D$ and 
 $\lambda= \tilde \lambda$.
\end{theorem}
Theorem \ref{thm1} says that the full knowledge of (harmonic combinations of) GPTs determines  uniquely the 
domain $D$ and the parameter $\lambda$.

Recall that the first-order polarization tensor,  $M_{\alpha \beta}$ with $|\alpha|= |\beta| =1$,  of any given inclusion is a $2\times2$  real valued and symmetric matrix.  Remarking that the polarization tensors  produced 
by rotating  ellipses with size coincide with the set of  $2\times2$   real valued symmetric matrices,  
it is known that the first-order polarization tensor yields the equivalent ellipse (see for instance \cite{AmmariKangb1} and references therein). The equivalent ellipse of $D$ is the ellipse with the same first-order polarization tensor as $D$. However, it is not known explicitly what kind of information on $D$ and $\lambda$ the higher-order GPTs carry. The purpose of this work is to study the possibility of using higher-order GPTs for  shape description. Our idea is based on first deriving a set of dense 
domains that can be identified from finitely many GPTs. We show that good candidates for
such a dictionary  is  the class $\mathcal G$ of  algebraic domains of size one.   Then, by approximating the target
domain using a sequence of algebraic domains in $\mathcal G$ we obtain a powerful tool for 
describing the shape of inclusions if  only a finite number of  GPTs is available. Next, we introduce the concept 
of real algebraic domains.
\section{Real algebraic domains}\label{sec3}
In the present section we consider the class of bounded open subsets of Euclidean space ${\mathbf R}^2$ with real algebraic 
boundary.  We adopt the following definition.

\begin{definition} \label{alg} An open set $G$ in ${\mathbf R}^2$ is called real algebraic if there exists a finite number of real coefficient polynomials 
$g_i(x), i=1,\cdots, m$ such that 
\bea
\p G \subset V:=\{x\in \R^2\;: g_1(x)= \cdots = g_m(x) =0\}.
\eea
 \end{definition}
 
The ellipse is a simple example of an algebraic domain, since its general boundary coincides with
the zero set of the quadratic polynomial function \[g(x)=  \sum_{|\alpha| \leq 2}g_\alpha x^\alpha \]
for given real coefficients $(g_\alpha)_{{|\alpha| \leq 2}}$ and proper signs in the top degree part.

We further denote by $\mathcal G$ the collection of bounded algebraic domains and {\it of size of order one}. The differential structure of the boundary $\p G$ is then well known:  
 it consists of algebraic arcs joining finitely many singular points, see for instance \cite{delaPuente}.
 
 It is tempting to also impose
 connectedness of the respective sets, but this constraint is not accessible by the elementary linear algebra tools we 
 develop in the present note, so we drop it. However, we call "domains" all elements $G \in \mathcal G$.
 
Following \cite{lp} we focus on a particular class of algebraic domains which are better adapted to the uniqueness 
and stability of the shape inverse quest. Let
 
\bean \label{set}
\mathcal G^* := \left \{G \in \mathcal  G:  G = \textrm{int}\, \overline G  \right \}.
\eean
 An element of $G^\ast$ is called an {\it admissible domain}, although it may not be connected.
 
The assumption that  $G = \textrm{int}\, \overline G$ implies that $G$ contains no slits
or $\p G$ does not have isolated points.
If $G\in \mathcal G^*$, the algebraic dimension of $\p G$ is one, and the ideal associated to it is 
principal. To be more precise, $\p G$ is a finite union of irreducible algebraic sets $X_j, \ j \in J,$ of dimension one each.
The reduced ideal associated to every $X_j$ is principal:
$$ I(X_j) = (P_j), \ \ \ j \in J$$
for instance  see \cite[Theorem 4.5.1]{bcr}. We assume that each $P_j$ is indefinite, i.e. it changes sign when crossing $X_j$.
Therefore one can consider the 
polynomial $g = \prod_{j \in J} P_j$, vanishing of the first-order on  $\p G$, that is
$|\nabla g| \not= 0$ on the regular locus of $\p G$. According to the real version of Study's lemma (cf. Theorem 12 in \cite
{delaPuente}) every polynomial vanishing on $\partial G$ is a multiple of $g$, that is 
$I(\p G) = (g)$. We define the degree of $\p G$ as the degree of 
the generator $g$ of the ideal $I(\p G)$. For a thorough discussion of the reduced ideal of a real algebraic surface
in ${\mathbf R}^d$, we refer to \cite{DuboisEfroymson}.

In the sequel, we denote by $g(x)$ the single polynomial vanishing  on $\p G$ which is the
generator of $I(\p G)$ and satisfying the following normalization condition 
$g_{\alpha^*} =1$, where $\alpha^*= \max_{g_\alpha\not=0} \alpha$.

Assume  $G\in \mathcal G^*$. If  the degree $d$ of $\p G$ and moments (up to order 3d) of the 
Lebesgue measure on $G$ are known, then as shown in \cite{lp}, the coefficients of 
 $g$  of degree $d$ that vanishes on $\p G$ are uniquely determined (up to a constant). More precisely,  
 it is shown in \cite{lp} that $g$  is the generator of a one-dimensional kernel of a  matrix whose entries are 
 obtained from moments of the Lebesgue supported measure by $G$. That is, only finitely many  
 moments (up to order 3d) are needed  to recover the minimal degree polynomial vanishing on $G$. It turns out that computing $g$ reduces
  to a solving a system of linear equations.
  
We stress that the main result contained in the present article identifies  the minimal degree polynomial $g$, and not the exact
  boundary of the admissible domain $G$. To give a simple example,  consider the defining equation of the boundary
  \begin{equation} \label{g4} g(x_1,x_2) = (x_1^3-x_1)(x_2^3-x_2).\end{equation}
  The following algebraic domains
   \begin{equation} \label{g4b} \begin{array}{l}  G_1 = \big\{ (x_1,x_2), x_1 >0 \ {\rm or} \  x_2 >0 \big\},\\
   \nm
   G_2 = \big \{ \{ (x_1,x_2), x_1 >0 \ {\rm or} \  x_2 <0 \big \},\\
   \nm
     G_3 = \big\{ \{ (x_1,x_2), x_1 <0 \ {\rm or} \  x_2 <0 \big \},\\
     \nm
       G_4 = \big\{ \{ (x_1,y_2), x_1 <0 \ {\rm or} \  x_2 >0 \big\},\end{array} \end{equation}
      are all admissible domains, sharing the same minimal degree defining function
      of the boundary.
      
 Even when restricting the class of domains to those possessing an irreducible boundary, we may encounter
 pathologies. Without recalling cumbersome details, Example 28 in \cite{delaPuente} produces a series of polynomials
 whose zero sets may contain curves and isolated points, but some of the isolated points are not irreducible components.
 The connectedness of $G$ is also tricky, as  for instance the intricate nature of the topology of the zero set of a lemniscate reveals:
 the curve
 \bean \label{lemniscate}
 \prod_{j=1}^k \big( (x_1 - a_j)^2 + (x_2-b_j)^2 \big) = r,
 \eean
 has $k$ distinct connected components for small values of $r>0$, where the poles $(a_j,b_j)$ are mutually distinct.
    
The main objective of our article, comparable to that in \cite{lp}, is to isolate a finite pool of domains (we may call them bounded "chambers") from which we can select the shape of $G$ and  further on determine the parameter  $\lambda$, both inferred from the knowledge of finitely many GPTs. It would be extremely interesting to unveil how the additional information encoded in the GPTs allows to select the correct
chamber among the many potential candidates.
    
 \section{Main results}\label{sec4}
 
Let $\R[x]$ be the ring of polynomials in the variables $x = (x_1, x_2)$ and let $\R_N[x]$
be the vector space of polynomials of degree at most $N$ (whose dimension is 
 $r_N=(N+1)(N+2)/2$). For a polynomial function $p(x) \in \R_N[x]$, it has a unique expansion
in the canonical basis $x^\alpha, |\alpha| \leq N$ of $\R_N[x]$, that is, 
\bea
p(x) = \sum_{ |\alpha| \leq N} p_\alpha x^\alpha 
\eea
for some vector coefficients ${\bf p} = (p_\alpha) \in \R_{r_N}$. The following is the main result of the paper.
\begin{theorem} \label{main}
Let $G \in \mathcal G^*$  with $\p G$  Lipschitz of degree $d$,  and let  $g(x)=  \sum_{ |\alpha| \leq d}g_\alpha x^\alpha$, be a polynomial function  that vanishes  of the first-order  on $\p G$,  satisfying $I(\p G) = (g),$
  $g_{\alpha^*} =1$, and $g(0)= 0$, where $\alpha^*= \max_{g_\alpha\not=0} \alpha$. Then,  there exists a discrete set $\Sigma \subset \mathbb C_0:=\mathbb C\setminus[-1/2, 1,2]$, such that 
${\bf g} = (g_\alpha) \in \R_{r_d}$ is the unique solution to the following normalized linear system:
\bean \label{gg1}
{\bf p} = (p_\alpha) \in \R_{r_d}; \;\;\sum_{|\beta| \leq d}M_{\alpha\beta}(\lambda, G)p_\beta = 0 \;\;  \textrm{for  } |\alpha| \leq 2d;\;\;
p_{\alpha^*} =1, \;\;\alpha^* =\max_{p_\alpha\not=0} \alpha \; \textrm{  for } \lambda \in \mathbb C_0\setminus \Sigma. 
\eean

\end{theorem}
\proof 

The proof of the theorem has two main steps. In the first step, we  show that ${\bf g}$ satisfies 
the normalized linear system \eqref{gg1}.  The second step consists in proving  that it is indeed the unique solution
to that system.
\begin{itemize}
\item[Step 1.]
For $|\lambda|> \frac{1}{2}$, recall  from \eqref{gpt1} the general expression of the GPTs 
\begin{equation} \label{gpt}
 M_{\alpha\beta}(\lambda, G) :=  \int_{\p G}  (\lambda I- \mathcal{K}^*_{\p G})^{-1} \left [\nu(x)\cdot\nabla x^\alpha\right ] y^\beta\,d\sigma(y).
 \end{equation}
 Since $g(x)$ vanishes on $\p G$, we infer 
 \begin{equation} \label{gg}
 \int_{\p G}  (\lambda I- \mathcal{K}^*_{\p G})^{-1} \left [\nu(x)\cdot\nabla x^\alpha\right ] g(y) \,d\sigma(y) =0, \qquad  \forall \alpha \in \N^2,
 \end{equation}
which combined with \eqref{gpt} leads to the desired system:
\bean \label{gg2}
\sum_{|\beta| \leq d}M_{\alpha\beta}(\lambda, G)g_\beta = 0, \qquad  \forall \alpha \in \N^2.
\eean
\item[Step 2.]
Assume that ${\bf p} \in \R_{r_d}[x]$ satisfies the system  \eqref{gg2}. Our objective is to prove that ${\bf p}$
coincides with  ${\bf g}$.

Denote by $\mathbb C_\star := \mathbb C \setminus (-\infty, -2]\cup [2, +\infty),$
and let 
$\mu= \lambda^{-1} \in \mathbb C_\star$. Define $\mathbb M(\mu) $ to be the rectangular matrix with
coefficients:
$(\mu M_{\alpha\beta})_{|\alpha| \leq 2d, |\beta| \leq d}$, that is 
\bean\label{gpt2}
 \mathbb M_{\alpha\beta}(\mu) :=  \int_{\p G}  ( I- \mu \mathcal{K}^*_{\p G})^{-1} \left [\nu(x)\cdot\nabla x^\alpha\right ] y^\beta\,d\sigma(y).
 \eean

Obviously,  the following equality holds \bea \label{eM}
 \mathbb M(\mu) {\bf p}= 0.
 \eea
 
 \end{itemize}
\begin{lemma} \label{holo} The function
$ \mu\rightarrow \mathbb M(\mu) \in \mathcal L\left(\mathbb R^{r_d}, 
\mathbb R^{r_{2d}} \right)$  is a holomorphic matrix-valued function on 
$\mathbb C_\star,$ and  $
\textrm{ker} (\mathbb M(0)) = \left\{c {\bf g}; \;\; c\in \mathbb R \right\}$.
\end{lemma}

\proof
Considering the properties of the resolvent set in Proposition~\ref{resolvent}, $\mu \mathcal{K}^*_{\p G}$ is a 
contraction operator for $\mu$ small enough.  In fact it can be easily verified that 
$\| \mathcal{K}^*_{\p G}\| \leq 1/2$ where the norm of $\mathcal{K}^*_{\p G}$ is taken in the energy space~\cite{KhavisonPutinarShapiro}.  Moreover, the Neumann series 
$$
 ( I- \mu \mathcal{K}^*_{\p G})^{-1} =  \sum_{j=0}^\infty \mu^j (\mathcal{K}^*_{\p G})^j,
 $$
has a holomorphic extension in  the resolvent set $\mathbb C_\star$. Next, we investigate the kernel of  $\mathbb M(0)$.
More precisely, $\mathbb M(0){\bf p}= 0$ is equivalent to 
\bean\label{gpt0}
\int_{\p G}  \nu(y) \cdot\nabla q(y)  p(y) \,d\sigma(y) = 0,  \quad \forall q \in  \R_{d}[x].
 \eean
Since $g$  generates the ideal associated to $\partial  G$ and $\p G$ is Lipschitz, we have 
 $\nu(x) = \frac{\nabla g}{\|\nabla g\|}$ on the regular part of the curve $\p G$.
 Then \eqref{gpt0} becomes 
\bea
\int_{\p G}  \nabla g(y) \cdot\nabla  q(y)  p(y)   \frac{1}{\|\nabla g\|}  \,d\sigma  = 0,  \quad \forall  q \in  \R_{d}[x].
 \eea
 
By taking $q(x)= x^\alpha g(x)$ for $\alpha \in \mathbb N^2$ satisfying $|\alpha| \leq  d$, and considering the
fact that $g(x)$ vanishes on $ \partial G$, one finds 
 
 \bea
\int_{\p G}\|\nabla g\| y^\alpha p(y)  \,d\sigma = 0,  \quad \forall  \alpha \in \mathbb N^2, \; \; |\alpha| \leq  d,
 \eea
 which in turn implies that

 \bea
\int_{\p G}\|\nabla g\|  q(y) p(y)  \,d\sigma = 0,  \quad \forall  q \in  \R_{d}[x].
 \eea
 Then, taking $q=p$ in the last inequality gives $p(y) = 0$ on $\partial G$. 
 Consequently, $p = c g$ for some real  constant $c$, which is the
 desired result.
\endproof

At this point we return to the proof of the main theorem. The analytic family of matrices $M(\mu)$ annihilates $g$ for all values of the parameter $\mu$. Moreover, for 
$\mu =0$ we saw that $M(\mu)$ has maximal rank, that is its kernel is spanned by $g$. Since maximal rank
is constant on a Zariski open subset of the parameter domain, we infer that 
$$ \dim \textrm{ker} M(\mu) = 1$$
for all $\mu$ in {$\mathbb C_\star$, except a discrete subset $\widetilde \Sigma$.}

 We provide some details of the proof for the convenience of the general readership. Let $\mathfrak H = \left\{c {\bf g}; \;\; c\in \mathbb R \right\}^\perp$  be  the sub-vector  
 space in  $\mathbb R^{r_d}$ orthogonal to $\textrm{ker} (\mathbb M(0))$. 
 Denote the restriction of $\mathbb M(\mu)$ to $\mathfrak H$ by
 $ \widetilde{\mathbb M}(\mu) \in \mathcal L(\mathfrak H, \R_{r_{2d}}) $. Since the Hilbert
 space $\mathfrak H $
 is $\mu$ independent, the matrix-valued function $\mu\rightarrow \widetilde{\mathbb M}(\mu)$ 
 inherits the same regularity as the function $\mu\rightarrow \mathbb M(\mu)$, i.e., it is holomorphic 
 on $\mathbb C_\star$. Then,  as a direct consequence of 
Lemma~\ref{holo}, the restriction of $\mathbb M(0)$ to $\mathfrak H$, denoted by
 $ \widetilde{\mathbb M}(0) \in \mathcal L(\mathfrak H, \R_{r_{2d}})$, is injective. Recall that 
a linear bounded operator is injective if and only if  it has a left inverse~\cite{Brezis}.
Then, there exists a left inverse denoted  by $L_0\in \mathcal L(\R_{r_{2d}}, \mathfrak H ) $,
that only depends on $\p G$ satisfying 
\bea
L_0  \widetilde{\mathbb M}(0)  = I_d,
\eea
where $I_{\mathfrak H}$ is the identity matrix acting on $\mathfrak H$.
Now, let
$$
\begin{array}{llcc}
\mathfrak H \longrightarrow  \mathfrak H\\
\mathbb T(\mu) =  L_0 \widetilde{\mathbb M}(\mu).
\end{array}
$$
Then, $\mu \rightarrow \mathbb T(\mu)$ is holomorphic on  $\mathbb C_\star$, and by construction 
it verifies  $\mathbb T(0) = I_{\mathfrak H}.$ Thus, we deduce from Steinberg Theorem  that 
$\mathbb T(\mu)$ is invertible everywhere on $\mathbb C_\star$  except at  a discrete set of values
$\widetilde \Sigma$~\cite{Kato}. Hence, $L_0$ becomes the left inverse of the 
 matrix $\widetilde{\mathbb M}(\mu)$  for all $\mu \in \mathbb C_\star \setminus \widetilde \Sigma$.
  Consequently,    $\widetilde{\mathbb M}(\mu)$   is injective 
  for all $\mu \in \mathbb C_\star \setminus \widetilde \Sigma$. 
  Since  $ \widetilde{\mathbb M}(\mu)$
  is the restriction of $\mathbb M(\mu)$ to $\mathfrak H$  which is the orthogonal space 
  to the vector $g$ in $\mathbb R^{r_d}$, we obtain that $\textrm{ker} (\mathbb M(\mu)) 
  = \left\{c {\bf g}; \;\; c\in \mathbb R \right\}$ for all $\mu \in \mathbb C_\star \setminus \widetilde \Sigma$.
   Then, for $\mu \in \mathbb C_\star \setminus \widetilde \Sigma$, \eqref{eM} implies
  ${\bf p} = c {\bf g} $ for some real constant $c$. Using the normalization condition, we obtain ${\bf p}={\bf g}$,
  and hence $p= g$ for all $\lambda \in \mathbb C\setminus[-1/2, 1,2]$ outside the set  
  $\Sigma := \{\lambda \in \mathbb C;\;\lambda^{-1} \in  \widetilde \Sigma \}$.
  
\endproof

\begin{remark} 
We propose here a direct proof of  the
uniqueness result  in  Theorem~\ref{thm1} in the case where 
all the  GPTs of the
algebraic domain $G$ are known. In \cite{ak03}, the proof of uniqueness for general
shapes is based  on the relation between
the far-field expansion and the Dirichlet-to-Neumann operator. 

Let ${\bf p^e}$ be  the extension of the 
vector ${\bf p} \in \R_{r_d}[x]$ by zero in  
the canonical basis $x^\alpha, |\alpha| > d$.  
Assume that ${\bf p} \in \R_{r_d}[x]$ satisfies the system  \eqref{gg1},
then 
\bea
\sum_{\beta}M_{\alpha\beta}(\lambda, G)p_\beta^e = 0, \qquad  \forall \alpha \in \N^2,
\eea
and consequently, 
\bea
 \int_{\p G}  (\lambda I- \mathcal{K}^*_{\p G})^{-1} \left [\nu(x)\cdot\nabla x^\alpha\right ]
 p(y) \,d\sigma(y) =0, \qquad  \forall \alpha \in \N^2.
\eea
Similarly, we have 
\bea
 \int_{\p G}  \nu(y)\cdot\nabla y^\alpha
  (\lambda I- \mathcal{K}_{\p G})^{-1} \left[p(x)\right] \,d\sigma(y) =0, \qquad  \forall \alpha \in \N^2, 
\eea
where $\mathcal{K}_{\p G}: \, H^{1/2}(\partial D) \rightarrow H^{1/2}(\partial D)$ denotes the adjoint of $\mathcal{K}^*_{\p G}$. 
Since $\Gamma (z-y)$ has the following expansion 
\bea
\Gamma (z-y) = \sum_{|\alpha|=0}^{\infty}\frac{(-1)^{|\alpha|}}{\alpha!}\p^\alpha 
\Gamma(z)y^\alpha, \quad  y\in \overline G, \quad |z| \rightarrow +\infty,
\eea
we deduce that
\bea
 \int_{\p G}  \nu(y)\cdot\nabla \Gamma (z-y)
  (\lambda I- \mathcal{K}_{\p G})^{-1} \left[p(x)\right] \,d\sigma(y) =0, \qquad  \forall z\in \R^2\setminus 
  \overline G,
\eea
which implies that
 \bea
 (\lambda I- \mathcal{K}_{\p G})^{-1} \left[p(x)\right] = 0 \qquad \textrm{on } \partial D.
\eea
Since $|\lambda|> \frac{1}{2}$, $\lambda I- \mathcal{K}_{\p G}$ is invertible and so
 $p$ vanishes completely on $\p G$. Consequently, $p = c g$ for some real constant $c$.
Using the normalization condition, we obtain $p=g$.
\end{remark}
{Following the discussion in Section~\ref{sec3}, we need to add a supplementary criteria 
to be able to identify uniquely the domain $G$ from its minimal polynomial $g$. Let  $\Omega_r$
be a bounded domain in $\mathbb R^2$, containing a ball of center zero and radius $r>0$ large
enough. Let $\mathcal G^*_0$ be the set of polynomial functions $g \in \mathcal G^*$ such 
that there exists a unique  Lipshitz algebraic domain $G$ containing zero and with size one satisfying
$\p G \subset \{g=0\}\cap \Omega$.}

\begin{corollary} \label{main2}
Let $G \in \mathcal G^*_0,  \widetilde G \in \mathcal G^*$  with respectively  $\p G$ and $\p \widetilde G$ of degree $d$. Let  $g(x)=  
\sum_{ |\alpha| \leq d}g_\alpha x^\alpha $ and $ \tilde g(x)=  \sum_{ |\alpha| \leq d}\tilde g_\alpha x^\alpha$ be  
respectively  polynomial functions  that vanish respectively  of the first-order  on $\p G$,  and $\p \widetilde G$ 
satisfying $I(\p G) = (g),\; g_{\alpha^*} =1$, \;$g(0)=0$,\, and\, $I(\p \widetilde G) = (\tilde g),\; \tilde g_{\alpha^*} =1,$
\; $\tilde g(0) = 0$. 
Let $\lambda   $  be fixed  in $\mathbb C_0$ such that $\lambda \notin \Sigma$,
where the set $\Sigma(\partial G)$ is as defined in Theorem~\ref{main}.
 Then,  the following uniqueness result holds:
\bean \label{uniqueness}
(M_{\alpha\beta}(G, \lambda))_{|\alpha| \leq 2d, 0<|\beta| \leq d} = (M_{\alpha\beta}(\widetilde G, 
\lambda))_{|\alpha| \leq 2d, 0<|\beta| \leq d}
\;\;\;\; \textrm{iff}\;\;\;\; G= \widetilde G.
\eean
 
\end{corollary}
\begin{proof}
The result is a direct consequence of Theorem~\ref{main}. Since the generalized polarization tensors coincide, 
and $\lambda \notin \Sigma$, we can deduce from Theorem~\ref{main} that $g= \tilde g$. Since $g\in \mathcal G^*_0$ 
we finally obtain  $G=\widetilde G$.
\end{proof}

\begin{remark}
From applications point of view, the assumption  $G \in \mathcal G^*_0$  in Corollary  \ref{main2},  is  
somehow related to the fact that in the inverse problem of identifying small inclusions from boundary voltage measurements, the location and the convex hull   
are well determined \cite{ak11}. Our method will allow the recovery of the shape up to a certain  precision 
fixed by the highest order of the considered  GPTs. The approach  can be seen as an extension of the equivalent 
 ellipse  approach \cite{AmmariKangb1}.  The assumption $G \in \mathcal G^*_0$ can be dropped when 
 the regular locus of $\p G$ coincides with  $\p G$. For example, when $G$ is a lemniscate
with a large enough level set constant  such that $\p G$ contains all the
complex roots of  its associated complex polynomial (for $r$ large enough in \eqref{lemniscate}). 
 
\end{remark}

\begin{remark}
In Theorem~\ref{main}, $\lambda$ is used implicitly. However, if the domain $G$ is sufficiently well approximated by $\tilde{G}$ then one can also attempt to recover $\lambda$ by solving the following minimization problem:
$$
\tilde{\lambda} = argmin_{}\|M_{\alpha\beta}(\lambda, G)-M_{\alpha\beta}(\tilde{\lambda}, \tilde{G})\|~;~\|M_{\alpha\beta}\| \neq 0 .
$$
\end{remark}

\section{Approximation by algebraic domains} \label{sec6}

For possibly non algebraic boundaries $\p G$, we describe a simple procedure to compute 
a polynomial $g$ whose level set $\{x : g(x) = 0\}$ approximates $\p D$. We expect better approximation to be found among higher degree polynomials. Domains enclosed by real algebraic curves (henceforth simply called algebraic domains) are dense, in Hausdorff metric among all planar domains. A very particular case is offered by domains surrounded by a smooth curve. They can be approximated by a sequence of algebraic domains. This observation turns algebraic curves into an efficient tool for describing shapes \cite{FatemiAminiVetterli, KerenCooperSubrahmonia, TaubinCukiermanSulliven}. Note that an algebraic domain which in addition {\it is the sub level set of a polynomial of degree $d$} can be determined from its set of two-dimensional moments of order less than or equal to $3d$ \cite{ lp}. On a related topics, an exact reconstruction of quadrature domains for harmonic functions was proposed in \cite{GustafssonMilanfarPutinar} with the advantage of providing a potential type function (similar to the ubiquitous barrier
method in global optimization), which detects the boundary of any complicated shape without having to make a choice among different potential chambers. More details about the approximation theory concepts related to this framework can be found in \cite{GP}.

{Theorem~\ref{main} suggests a strategy to approximately recover information on the boundary $\p D$ when the latter is not algebraic. We will follow the approach developed  in  \cite{lp} with interior moments. By  considering the GPTs 
$(M_{\alpha\beta})_{|\alpha| \leq 2d, |\beta| \leq d}$, one may compute the polynomial $g\in \R_{d}[x]$ with coefficients ${\bf g} \in \R_{r_d}$ such that ${\bf g}$ is  the most suitable singular vector corresponding to the  smallest in absolute value singular value of $(M_{\alpha\beta})_{|\alpha| \leq 2d, |\beta| \leq d}$. }


Building on this, we also suggest an algorithm for shape recognition. The algorithm has three steps: (i) recovering $\textbf{g}$ for some degree $d$ polynomial; (ii) checking to see if the recovered polynomial has bounded level set, and (iii) accounting for scaling and rotations via minimization problem.

We stress that not all simple real algebraic sets in ${\mathbf R}^2$, such as a triangle, are the sublevel set of a single polynomial. The reader should be aware that the algorithms below identify the minimal polynomial $g$ vanishing on the boundary of an admissible planar domain $G$, but by no means this implies $$ G = \{ x \in {\mathbf R}^2; \ g(x) >0 \}.$$ Even worse, our numerical procedure of
plotting the potential boundary of $G$ is based on a curve selection process, which may not detect in a single shot all irreducible components of $\p G$. The intricate details of amending these weak points of our numerical schemes and experiments will be addressed in a forthcoming article.

 We start by showing how orthogonal transformations on the domain relate to linear operation on the underlying coefficients. We use this relation to define a minimization scheme for finding the similarity between reference and target polynomials. Start by writing the algebraic domain as follows:
$$
G :=\{x \in \mathbb{R}^2~:~ g(x) = \textbf{x}^t_{[k]}\textbf{g}_{[k,k]}^2\textbf{x}_{[k]}+\sum_{j=0}^{2k-1}\textbf{g}_{[j]}\textbf{x}_{j}=\sum_{j=0}^{2k}\textbf{g}_{[j]}\textbf{x}_{[j]} = 0 \},
$$
where $g$ is a polynomial of even degree $d =2k$ (necessary) written as a sum of polynomial forms (homogeneous) and the superscript $t$ denotes the transpose.  With the notation explained by 
\begin{definition} Let
$\textbf{x}_{k}^t := [x_1^{k},x_1^{k-1}x_2^{1},...,x_1^{1}x_2^{k-1},x_2^{k}]$, $\textbf{g}_{[k,k]}$ be a $(k+1)-by-(k+1)$ matrix, and $\textbf{g}_{[j]}$ be a row vector of $j+1$ real coefficients.  Note that $\textbf{g}_{[2k]}$ represents the coefficients of the leading form and $\textbf{g}_{[k,k]}$ is its quadratic form\cite{TaubinCukiermanSulliven}.
\end{definition}
The ability to write any polynomial in this form is a consequence of Euler's theorem \cite{TaubinCukiermanSulliven}. This formulation is used in \cite{TaubinCukiermanSulliven} to prove the following results. 
\begin{lemma}\label{boundedpoly} 
If $\textbf{g}_{[k,k]}$ is non-singular, then $G$ is bounded and non-empty. 
\end{lemma}
\begin{lemma}\label{boundedpoly2} 
All odd degree forms have unbounded level sets.
\end{lemma}
In order to recover shapes, we must first define what it means for two shapes to be the same. We avoid the difficulty of defining a similarity measure and simply state that a shape should be invariant under rotations and scaling. This invariance takes the form of a matrix operation on the underlying coefficients. Let $\mathcal{O}_s(2):=\{ A= s R \mbox{ with $R$ being a rotation matrix and $s>0$} \}$. The following result holds. 

\begin{lemma}
Consider a transformation $A=(a_{i,j})_{i,j \in \{1,2\} } \in \mathcal{O}_s(2)$. Let $x^\prime=Ax$. Then 
$$
\textbf{x}_{[d]}^\prime=A_{[d]}\textbf{x}_{[d]}
$$
and $$A(G) = \{x \in \mathbb{R}^2~:~ g(x)  =\sum_{j=0}^{2k}\textbf{g}_{[j]}A_{[j]}^{-1}\textbf{x}_{[j]} = 0 \}$$
with the $(d+1) \times (d+1)$ matrix $A_{[d]}$ having entries given by
$$
(A_{[d]})_{h,j\leq d} =\sum_{p=0}^{j}c_p(h) d_{j-p}(h)~;~c_j(h) = {{d-h}\choose{j}}a_{1,1}^{d-h-j}a_{1,2}^j ~:~d_j(h)={{h}\choose{j}}a_{2,1}^{h-j}a_{2,2}^j .
$$
Here, $c_j$ and $d_j$ are extended by zeros to be defined up to $j \leq d$.
\end{lemma}
\begin{proof} We write
$$\textbf{x}^{\prime t}_{[d]} = [(a_{1,1}x_1+a_{1,2}x_2)^{d},(a_{1,1}x_1+a_{1,2}x_2)^{d-1}(a_{2,1}x_1+a_{2,2}x_2),...,(a_{2,1}x_1+a_{2,2}x_2)^{d}].
$$
For the $h$-th entry, we have the following:
$$
(a_{1,1}x_1+a_{1,2}x_2)^{d-h}(a_{2,1}x_1+a_{2,2}x_2)^h=\Big{(}\sum_{j=0}^{d-h}{{d-h}\choose{j}}a_{1,1}^{d-h-j}a_{1,2}^jx_1^{d-h-j}x_2^{j}\Big{)}\Big{(}\sum_{j=0}^{h}{{h}\choose{j}}a_{2,1}^{h-j}a_{2,2}^jx_1^{h-j}x_2^{j}\Big{)}
$$
$$
=\Big{(}x_1^{d-h}\sum_{j=0}^{d-h}c_jy^j\Big{)}\Big{(}x_1^{h}\sum_{j=0}^{h}d_jy^j\Big{)}= \sum_{j=0}^{d}\Big{(}\sum_{p=0}^{j}c_p(h) d_{j-p}(h) \Big{)}x_1^{d-j}x_2^{j}=: A_{[d],h}\textbf{x}_{[d]} ,
$$
where $y^j:=x_1^{-j}x_2^{j}$, $c_j(h) = {{d-h}\choose{j}}a_{1,1}^{d-h-j}a_{1,2}^j$ and $d_j(h)={{h}\choose{j}}a_{2,1}^{h-j}a_{2,2}^j$ are extended by zeros to be defined up to $j \leq d$. 
\end{proof}


With this result, we can first suggest a simple algorithm for matching GPTs to a predefined real algebraic shape. The goal is to recover the matrix $A$ as it gives the relation between reference and observed shapes. 

\begin{algorithm}
  \caption{Shape matching}
  \label{algo:kalman}
  \begin{algorithmic}
    \STATE {Input:} The GPTs $(M_{\alpha\beta})_{|\alpha|\leq 2d,|\beta|\leq d}$ and the coefficients $\textbf{g}$ of the reference shape.
\STATE {Procedure:} \begin{itemize}
\item Solve $\tilde{\textbf{g}}$  using Theorem~\ref{main}; 
\item Check if Lemma \ref{boundedpoly} holds for $\tilde{\textbf{g}}$;
\item  Find $\tilde{A} := argmin_{A \in \mathcal{O}_s(2)} \|\sum_{j=0}^{2k}(\tilde{\textbf{g}}_{[j]}-\textbf{g}_{[j]}A_{[j]}) \|$ and $\epsilon_{match}:= \| \sum_{j=0}^{2k}(\tilde{\textbf{g}}_{[j]}-\textbf{g}_{[j]}\tilde{A}_{[j]}) \|$. 
\end{itemize}
\STATE {Output:} $\tilde{A}$ and $\epsilon_{match}$.
\end{algorithmic}
\end{algorithm}

\begin{example}(\textbf{Shape search real algebraic domains}). The following 6 plots show real algebraic domains with degrees ranging from 6 to 2. The GPTs were calculated from algebraic domains and then  the algebraic domains were recovered again via Theorem~\ref{main}. The first shape is a reference shape, see Figure \ref{example31}. 
\begin{figure}[h]
\includegraphics[scale=0.6]{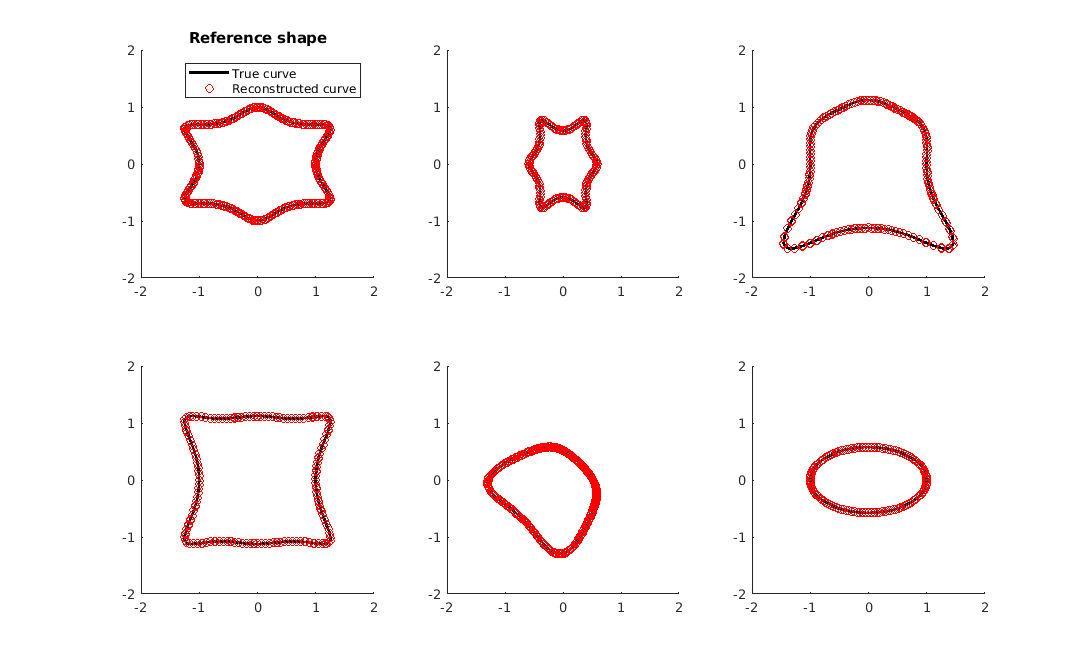}
\caption{Reconstruction of real algebraic shapes. \label{example31}}
\end{figure}

 The coefficients are minimized under rotations and scaling. The algorithm correctly matches the second shape. 
 
\end{example}

In the general case, we propose a shape reconstruction algorithm for not necessary algebraic domains.
\begin{algorithm}
  \caption{Approximating general domains with algebraic ones}
  \label{algo:kalman2}
  \begin{algorithmic}
    \STATE {Input:} The GPTs
 $(M_{\alpha\beta})_{|\alpha| \leq 2d, |\beta| \leq d}$ obtained from a non-algebraic domain.
\STATE {Procedure:} \begin{itemize}
\item Let $\beta^{(1)},\beta^{(2)},\ldots,\beta^{(r_d)}$ be an enumeration of the $\beta$'s;
\item Let $\alpha^{(1)},\alpha^{(2)}, \ldots,\alpha^{(r_{2d})}$ be an enumeration of the $\alpha$'s;
\item Construct a matrix $[L]_{1 \leq i \leq r_{2d},   1 \leq j \leq r_{d} }:=M_{\alpha^{(i)}\beta^{(j)}}$;
\item Find the singular vector $[\tilde{g}^{(1)}, \ldots, \tilde{g}^{(r_d)}]^t$ corresponding to the smallest in absolute value singular value of $L$;
\item Set $\tilde{g}^{(j)} = \tilde{g}_{\beta^{(j)}}$.
\end{itemize}
\STATE {Output:} The coefficients $\tilde{\textbf{g}}$ of the reconstructed real algebraic shape.
\end{algorithmic}
\end{algorithm}

\begin{example}(\textbf{Approximating non-algebraic domains})
As stated earlier, algebraic domains can approximate any planar domain. However, some shapes may require an infinite series of polynomials to be described. We approximate the shapes of a triangle, a diamond and a flower with one petal missing. The GPTs were calculated from these diametrically defined shapes which in turn was used to  recover equivalent-polynomials, see Figure \ref{example32}.

\begin{figure}[h]
\includegraphics[scale=0.5]{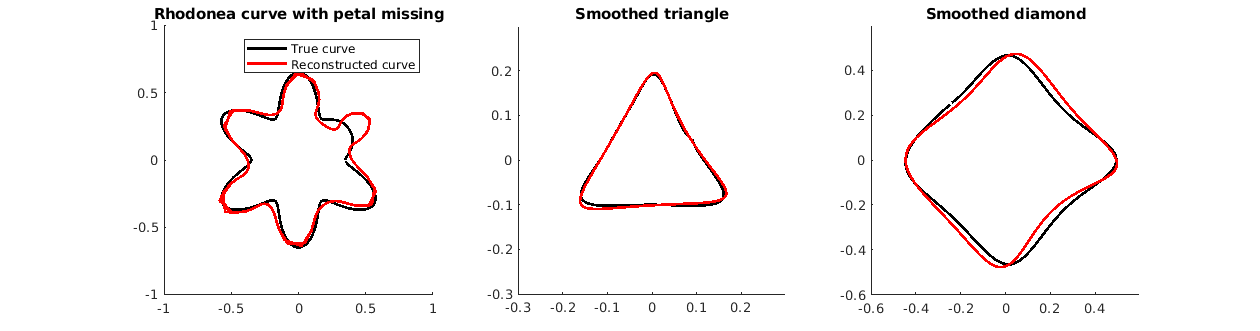}
\caption{Reconstruction of non-algebraic shapes. \label{example32}}
\end{figure}

\end{example}

\begin{remark}
Note that a higher degree polynomial may not always yield a better approximation. 
\end{remark}

\begin{example}(\textbf{Recovering non-connected real algebraic domains})
The proposed reconstruction method is powerful enough to reconstruct non-connected domains. Consider the polynomial given in (\ref{g4}). Below, we plot the real level sets of (\ref{g4b}) for values $-0.1,~0$ and $0.1$. On the $0.1$-level two components of the domain are selected and their GPTs are computed. From these GPTs, a polynomial is recovered via Algorithm \ref{algo:kalman2}. The real $0$-level set of the recovered polynomial contains curves similar to the ones  used to compute the GPTs but also  unbounded components, see Figure \ref{example33}.

\begin{figure}[h]
\includegraphics[scale=0.5]{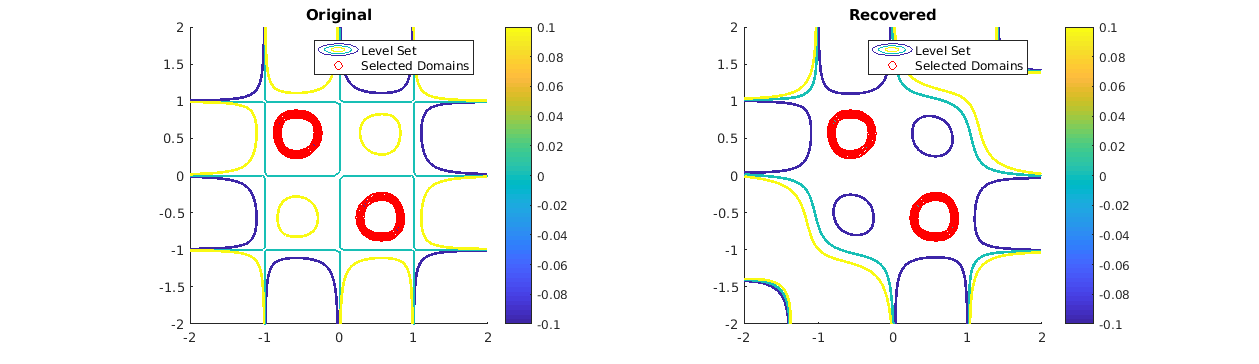}
\caption{Reconstruction of non-connected domains. \label{example33}}
\end{figure}
\end{example}

\section{Conclusion}\label{sec7}
In this paper, we have introduced a new tool for identifying shapes from finite numbers of their associated GPTs. 
We have shown in Theorem~\ref{main} and Corollary~\ref{main2} that  the knowledge of the $r_d \times r_{2d}$ GPTs 
 of a real algebraic domain having a boundary of degree $d$ is sufficient to identify it.  This is a very promising
path since it is now  possible to recover intricate  shapes  with the knowledge 
of only a finite number of GPTs.  These results confirm  that  GPTs are good  shape descriptors and can be harnessed to great effect by even simple algorithms.  
We believe that the number $r_d \times r_{2d}$ of required GPTs  can be dramatically  reduced, and this will be the subject of a future  work.




\begin{thebibliography}{99}




\bibitem{AmmariBoulierGarnierJingKangWang}
  {\sc H. Ammari, T. Boulier, J. Garnier, W. Jing, H. Kang and H. Wang, }
  {\em Target identification
    using dictionary matching of generalized polarization tensors,}
    Found. Comput. Math., 14 (2014), no. 1, 27--62.
    
\bibitem{AmmariCiraoloKangLeeMilton}
{\sc H. Ammari, G. Ciraolo, H. Kang, H. Lee and G. Milton,}
{\em Spectral theory of a Neumann-Poincar\'e-type operator
and analysis of cloaking due to anomalous localized resonance,} 
{Arch. Ration. Mech. Anal., 208 (2013), 667--692.}



\bibitem{AGKLY} 
{\sc H. Ammari, J. Garnier, H. Kang, M. Lim, and S. Yu,}
{\em Generalized polarization tensors for shape description,}
 Numer. Math., 126 (2014), no. 2, 199--224. 


\bibitem{cloaking} 
 {\sc H. Ammari, H. Kang, H. Lee, and M. Lim,} {\em Enhancement of near cloaking using generalized polarization tensors vanishing structures: Part I: The conductivity problem},  Comm. Math. Phys., 317 (2013), no. 1, 253--266.
 
 
\bibitem{ak03} 
{\sc H. Ammari and H. Kang,}
{\em  Properties of the generalized polarization tensors,}
 SIAM J. Multiscale Modeling and Simulation, 1 (2003), 335--348.



\bibitem{AmmariKangb1} 
{\sc H. Ammari and H. Kang,} 
{\em Polarization and moment
tensors with applications to inverse problems and effective medium
theory}, Applied Mathematical Sciences, Vol. 162, Springer-Verlag,
New York, 2007.

\bibitem{ak11} 
{\sc H. Ammari and H. Kang,}
{\em Expansion methods,} { Handbook of Mathematical Methods of Imaging},
447-499, Springer, 2011.

\bibitem{AKLZ} {\sc H. Ammari, H. Kang, M. Lim, and H. Zribi, }
{\em The generalized polarization
tensors for resolved imaging. Part I: Shape reconstruction of a
conductivity inclusion, }
Math. Comp., 81 (2012), 367--386.


\bibitem{plasmonics} {\sc H. Ammari, P. Millien, M. Ruiz, and H. Zhang,} {\em Mathematical analysis of plasmonic nanoparticles: the scalar case},  Arch. Ration. Mech. Anal., 224 (2017), no. 2, 597--658. 

\bibitem{jmpa} {\sc H. Ammari, M. Putinar, M. Ruiz, S. Yu, and H. Zhang},
Shape reconstruction of nanoparticles from their associated plasmonic resonances,
J. Math. Pures et Appl., DOI: https://doi.org/10.1016/j.matpur.2017.09.003.


\bibitem{AndoKang} 
{\sc K. Ando and H. Kang.} 
{\em Analysis of plasmon resonance on smooth domains using spectral 
properties of the Neumann-Poincar\'e operator,}
{J. Math. Anal. Appl., 435 (2016) 162--178.}


\bibitem{bcr} 
{\sc J. Bochnak, M. Coste, and  M.-F. Roy,}
{\em  Real Algebraic Geometry},    Springer, Berlin, 1998.


\bibitem{BDT}
{\sc E. Bonnetier, C. Dapogny, and F. Triki},
{\em Homogenization of the eigenvalues of the Neumann-Poincar\'e operator.}
{Preprint (2017).}

\bibitem{BonnetierTriki}
{\sc E. Bonnetier, and F. Triki,}
{\em Pointwise bounds on the gradient and the spectrum of the
Neumann-Poincare operator: The case of 2 discs,}
{in H. Ammari, Y. Capdeboscq, and H. Kang (eds.) 
Conference on Multi-Scale and High-Contrast PDE: 
From Modelling, to Mathematical Analysis, to Inversion. 
University of Oxford, AMS Contemporary Mathematics, 577 (2012) 81--92.}

\bibitem{BonnetierTriki_2} 
{\sc E. Bonnetier and F. Triki,}
{\em On the spectrum of the Poincar\'e variational problem 
for two close-to-touching inclusions in 2d,}
{Arch. Rational Mech. Anal., 209 (2013), 541--567.}

\bibitem{Brezis}
{\sc H. Brezis,}
 {\em  Functional analysis, Sobolev spaces and partial differential equations},
  Universitext. Springer, New York, 2011.
 
\bibitem{DuboisEfroymson}
{\sc D. Dubois and G. Efroymson},
{\em Algebraic theory of real varieties}, in vol.
{\it Studies and Essays presented to Yu-Why Chen on his 60-th birthday},
Taiwan University, 1970, pp. 107-135.
 
 \bibitem{seo} {\sc E. 
  Fabes, M. Sand, and J.K. Seo}, {\em The spectral radius of the classical layer potentials on convex domains}. Partial differential equations with minimal smoothness and applications (Chicago, IL, 1990), 129--137, IMA Vol. Math. Appl., 42, Springer, New York, 1992. 
  
\bibitem{FatemiAminiVetterli} 
{\sc M. Fatemi, A. Amini, and M. Vetterli,}
{\em  Sampling and reconstruction of shapes with algebraic boundaries,}
 IEEE Trans. Signal Process, 64 (2016), no. 22, 5807--5818. 
 
\bibitem{GustafssonMilanfarPutinar}  
{\sc B. Gustafsson, C. He, P. Milanfar, and M. Putinar,}
{\em  Reconstructing planar domains from their moments, }
Inverse Problems, 16 (2000), 1053--1070.

\bibitem{GP} 
{\sc B. Gustafsson and M. Putinar}
{\em Hyponormal Quantization of Planar Domains},
Lect. Notes Math. vol. 2199, Springer, Cham, 2017.


\bibitem{Hu} {\sc M.-K. Hu,}
{\em Visual pattern recognition by moment invariants,}
 IRE Trans. Inform. Theory, 8 (1962), 179--187.



\bibitem{kang} 
{\sc H. Kang, }
{\em Layer potential approaches to interface problems, }
in {\sl Inverse problems and imaging}, 63--110, 
H. Ammari and J. Garnier, \'ed.,  
Panoramas and Syntheses 44, 
Soci\'ete Math\'ematique de France, Paris, 2015.

\bibitem{Kato} 
{\sc T. Kato,} 
{\em Perturbation theory for linear operators,}
 Springer Science \& Business Media, 2013.

\bibitem{KerenCooperSubrahmonia} 
{\sc D. Keren, D. Cooper, and J. Subrahmonia,}
{\em  Describing complicated objects by implicit polynomials, }
IEEE Trans. Pattern Anal. Mach. Intellig., 16 (1994), 38--52.


\bibitem{KhavisonPutinarShapiro}
{\sc D. Khavinson, M. Putinar, and H. S. Shapiro,}
{\em Poincar\'e's variational problem in potential theory,}
{Arch. Ration. Mech. Anal., 185, no. 1 (2007) 143--184.}

\bibitem{lp} 
{\sc J-B Lasserre, and  M. Putinar,}
{\em Algebraic-exponential Data Recovery from Moments,}
 Discrete \& Computational Geometry,  54 (2015), 993--1012.


\bibitem{milton} 
{\sc G.W. Milton,}
 {\em The Theory of Composites}, 
 Cambridge Monographs on Applied and
 Computational Mathematics, Cambridge University Press, 2002.

 \bibitem{delaPuente} 
 {\sc M. J. de la Puente}
 {\em Real Plane Alegbraic Curves,}
 Expo. Mathematicae, 20 (2002), 291--314.
 
\bibitem{TaubinCukiermanSulliven}
{\sc G. Taubin, F. Cukierman, S. Sulliven, J. Ponce, and D.J. Kriegman,}
{\em Parametrized families of polynomials for bounded algebraic curve and surface fitting, }
IEEE Trans. Pattern Anal. Mach. Intellig., 16 (1994), 287--303.

\bibitem{TrikiVauthrin}
{\sc F. Triki. and M. Vauthrin,}
{\em Mathematical modeling of the Photoacoustic effect generated by the heating of metallic nanoparticles, }
to appear in Quarterly of Applied Mathematics (2018).


\end{thebibliography}
\end{document}